\documentclass[11pt,reqno]{amsart}
\usepackage{indentfirst, amssymb,amsmath,amsthm, mathrsfs,cite,graphicx,float}
\newtheorem*{theoA}{Theorem A}
\newtheorem*{theoB}{Theorem B}

\newtheorem{theo}{Theorem}[section]
\newtheorem{lem}{Lemma}[section]

\newtheorem{ques}{Question}[section]

\newtheorem{open problem}{Open problem}[section]
\newcommand{\pa}{\partial}

\newcommand{\be}{\begin{equation}}
\newcommand{\ee}{\end{equation}}
\newcommand{\C}{\mathbb{C}}
\newcommand{\D}{\mathbb{D}}
\newcommand{\N}{\mathbb{N}}
\newcommand{\bs}{\begin{small}}
\newcommand{\es}{\end{small}}
\newcommand{\beas}{\begin{eqnarray*}}
\newcommand{\eeas}{\end{eqnarray*}}
\newcommand{\bea}{\begin{eqnarray}}
\newcommand{\eea}{\end{eqnarray}}
\renewcommand{\epsilon}{\varepsilon}
\numberwithin{equation}{section}

\begin{document}
\title[the Bohr inequality for Ces\'aro operator]
{Improved Bohr-type inequalities for the Ces\'aro operator}
\author[V. Allu, R. Biswas and R. Mandal]{ Vasudevarao Allu, Raju Biswas and  Rajib Mandal}
\date{}
\address{Vasudevarao Allu, Department of Mathematics, School of Basic Sciences, Indian Institute of Technology Bhubaneswar, Bhubaneswar-752050, Odisha, India}
\email{avrao@iitbbs.ac.in}
\address{Raju Biswas, Department of Mathematics, Raiganj University, Raiganj, West Bengal-733134, India.}
\email{rajibmathresearch@gmail.com}
\address{Rajib Mandal, Department of Mathematics, Raiganj University, Raiganj, West Bengal-733134, India.}
\email{rajubiswasjanu02@gmail.com}
\maketitle
\let\thefootnote\relax
\footnotetext{2020 Mathematics Subject Classification: 30C45, 30C50, 40G05, 44A55.}
\footnotetext{Key words and phrases: Bounded analytic functions, Improved Bohr radius, Bohr-Rogosinski radius, Ces\'aro operator.}
\footnotetext{Type set by \AmS -\LaTeX}
\begin{abstract}
In this paper, we derive the sharp improved versions of Bohr-type inequalities for the Ces\'aro operator acting on the class of bounded analytic functions defined on the unit disk
$\D=\left\{z\in\C:\left|z\right|<1\right\}$. 
In order to achieve these results, we utilize the principle of substituting the initial coefficients of the majorant series with the absolute values of the Ces\'aro operator associated with a 
bounded analytic function defined on $\D$ and its derivative, as well as for the Schwarz function.
\end{abstract}
\section{Introduction and Preliminaries}
Let $\D=\{z\in\C :|z|< 1\}$ denote the open unit disk in the complex plane $\C$ and $\mathcal{B}$ denote the class of analytic functions $f$ in $\D$ such that 
$|f(z)|\leq 1$ in $\D$. In 1914, H. Bohr \cite{6} proved that the majorant series $M_f(r)=\sum_{n=0}^\infty |a_n| r^n\leq 1$ for $r\leq 1/6$. Later, Weiner, Riesz, and Schur 
\cite{8a} independently improved the number $1/6$ to the best possible constant $1/3$. The radius $1/3$ is popularly known as the Bohr radius for the class $\mathcal{B}$. The 
result have also been proved by Sidon \cite{28} and Tomi\'c \cite{30}. Note that the inequality $M_f(r) \leq 1$ fails to hold for $r > 1/3$ for $f\in\mathcal{B}$. 
By considering $f_a(z) = (a- z)/(1-az)$, it is easy to see that $M_{f_a}(r)>1$ if, and only if, $ r>1/(1+2a)$, which shows that $1/3$ is the best possible for $a\to1^-$.\\[2mm]
\indent  Over the past two decades, researchers have demonstrated a notable interest in Bohr type inequalities. For a 
comprehensive review of this topic, we refer to \cite{1,2,301,302,303,304,305,4,8,12,306,307,14,16,308,R1} and the references cited therein. Boas and Khavinson \cite{7} 
have generalized the Bohr radius for the 
case of several complex variables by finding the multidimensional
Bohr radius. Many researchers have followed this paper to extend and generalize the phenomenon in different contexts (see \cite{3,21a,21b,25}).\\[2mm]
\indent Besides the Bohr radius, there is a notion of Rogosinski radius \cite{21,26} which is described as follows:
Let $f\in\mathcal{B}$ defined by $f(z)=\sum_{n=0}^\infty a_nz^n$ and $S_N(z):=\sum_{n=0}^{N-1}a_nz^n$ be the partial sum of $f$. Then, $|S_N(z)|<1$ for all $N\geq 1$ in 
the disk $|z|<1/2$. Here $1/2$ is sharp, known as Rogosinski radius. Motivated by the Rogosinski radius, Kayumov and Ponnusamy \cite{13} have introduced the Bohr-Rogosinski sum $R_N^f(z)$ for $f\in\mathcal{B}$, which is defined as 
\beas R_N^f(z):=|f(z)|+\sum_{n=N}^{\infty}|a_n||z|^n.\eeas
It is easy to see that $|S_N(z)|=\left|f(z)-\sum_{n=N}^{\infty}a_nz^n\right|\leq R_N^f(z)$. Moreover, the Bohr-Rogosinski sum $R_N^f(z)$ is related to the classical Bohr sum 
(Majorant series) in which $N =1$ and $f(0)$ is replaced by $f(z)$. Kayumov and Ponnusamy \cite{13} have defined the Bohr-Rogosinski radius for $f\in\mathcal{B}$ as the largest 
number  $r\in(0,1)$ such that $R_N^f(z)\leq 1$ for $|z| < r$. For a comprehensive study of the Bohr-Rogosinski radius, we refer to \cite{12,15,18} and the references cited therein.
The following question arises naturally.
\begin{ques}For certain complex integral operators defined on different function spaces, is it possible to extend the Bohr type inequality?
\end{ques}
The Bohr radius has been studied for the classical Ces\'aro operator in \cite{17,18} and for the Bernardi integral operator in \cite{20}, within the context of the unit disk $\D$.
A number of studies (see \cite{5,17,18,19,20, 101}) have been carried out on inequalities for the Bernardi integral operator, Ces\'aro operator and its various generalizations on the 
family $\mathcal{B}$ with the aim of establishing the Bohr type and Bohr-Rogosinski type inequalities.
Furthermore, the boundedness and compactness of the Ces\'aro operator in various function spaces have been thoroughly investigated. 
For $m\in\N$, let 
\beas \mathcal{B}_m=\left\{\omega\in\mathcal{B}: \omega(0)=\omega'(0)=\cdots=\omega^{(m-1)}(0)=0\quad\text{and}\quad\omega^{(m)}(0)\not=0\right\}.\eeas
In the context of the classical setting, for an analytic function $f(z)=\sum_{n=0}^\infty a_nz^n$ on the unit disk $\D$, the Ces\'aro operator \cite{8,11,29} is defined as
\bea\label{j2}\mathcal{C}f(z):&=&\sum_{n=0}^\infty \frac{1}{n+1}\left(\sum_{k=0}^n a_k\right) z^n=\int_{0}^1\frac{f(t z)}{1-t z} d t.\eea
If $f\in\mathcal{B}$, then the corresponding Bohr's sum \cite[P. 616]{17} is defined by
\bea\label{j3}\mathcal{C}_f(r):&=&\sum_{n=0}^\infty \frac{1}{n+1}\left(\sum_{k=0}^n |a_k|\right) r^n=\sum_{k=0}^\infty |a_k|\phi_k(r)\nonumber\\[2mm]
&=&\frac{|a_0|}{r}\log\frac{1}{1-r}+\sum_{k=1}^\infty |a_k|\phi_k(r),\eea
where 
\bea\label{j4} \phi_k(r)=\sum_{i=k}^\infty \frac{r^i}{i+1}=\frac{1}{r}\sum_{i=k}^\infty \left(\int_0^r x^i dx \right) =\frac{1}{r}\int_0^r \frac{x^k}{1-x} dx.\eea
Therefore, 
\beas \sum_{k=2}^\infty \phi_k(r)=\sum_{k=2}^\infty\left( \frac{1}{r}\int_0^r \frac{x^k}{1-x} dx\right)=\frac{1}{r}\int_0^r \frac{1}{1-x}\left(\sum_{k=2}^\infty x^k\right)dx=\frac{1}{r}\int_0^r \frac{x^2}{(1-x)^2} dx.\eeas
It is evident that for $f\in\mathcal{B}$, 
\beas \left|\mathcal{C}f(z)\right|\leq \frac{1}{r}\log\frac{1}{1-r} \quad\text{for}\quad |z|=r<1.\eeas 
In 2020, Kayumov {\it et al.} \cite{17} established the following Bohr type inequality for Ces\'aro operator.
\begin{theoA}\cite{17} If $f\in\mathcal{B}$ such that $f(z)=\sum_{n=0}^\infty a_nz^n$ for $z\in \D$. Then,
\beas \mathcal{C}_f(r)=\sum_{n=0}^\infty \frac{1}{n+1}\left(\sum_{k=0}^n |a_k|\right) r^n\leq \frac{1}{r}\log\frac{1}{1-r} \quad\text{for}\quad |z|=r<R,\eeas
where $R=0.5335...$ is the positive root of the equation 
\beas 2x-3(1-x)\log \frac{1}{1-x}=0.\eeas
The number $R$ is the best possible. 
\end{theoA}
 In 2021, Kayumov {\it et al.} \cite{18} established the following Bohr-Rogosinski radius for Ces\'aro operator on the space of bounded analytic function.
 \begin{theoB}\cite{18} If $f\in\mathcal{B}$ such that $f(z)=\sum_{n=0}^\infty a_nz^n$ for $z\in \D$. Then,
\beas \left|\mathcal{C}f(z)\right|+\sum_{k=N}^\infty |a_k|\phi_k(r)\leq \frac{1}{r}\log\frac{1}{1-r} \quad\text{for}\quad |z|=r<R_N,\eeas
where $\phi_k(r)$ is defined in (\ref{j4}) and $R_N$ is the positive root of the equation 
\beas 2r^{N+1}-(1-r)\log(1-r)-2Nr(1-r)\phi_N(r)=0.\eeas
The number $R_N$ is the best possible. 
\end{theoB}
\noindent The following are key lemmas of this paper and will be used to prove the main results.
\begin{lem}\label{lem1} \cite[Pick’s invariant form of Schwarz’s lemma]{201} Suppose $f$ is analytic in $\D$ with $|f(z)|\leq1$, then 
\beas |f(z)|\leq \frac{|f(0)|+|z|}{1+|f(0)||z|}\quad\mbox{for}\quad z\in\D.\eeas\end{lem}
\begin{lem}\cite{7a,26a}\label{lem2} Suppose $f$ is analytic in $\D$ with $|f(z)|\leq1$, then we have 
\beas \frac{\left|f^{(n)}(z)\right|}{n!}\leq \frac{1-|f(z)|^2}{(1-|z|)^{n-1}(1-|z|^2)}\quad\text{and}\quad |a_n|\leq 1-|a_0|^2\quad\text{for}\quad n\geq 1\quad \text{and}\quad |z|<1.\eeas\end{lem}
\section{Main results and their proofs}
 The following result is the sharp improved version of the Bohr inequality for Ces\'aro operator on the space of bounded analytic function.
\begin{theo}\label{T1}
 If $f\in\mathcal{B}$ such that $f(z)=\sum_{n=0}^\infty a_n z^n$ for $z\in\D$, then 
 \bea\label{j1}\left|\mathcal{C}f(z)\right|+\left|\mathcal{C}f'(z)\right|\phi_1(r)+\sum_{k=2}^{\infty}\left|a_k\right|\phi_k(r)\leq \frac{1}{r}\log \frac{1}{1-r}\eea
 for $|z|= r\leq r_1\leq R(\approx 0.493411)$, where $r_1\in(0,1)$ is the unique positive root of the equation
 \beas&& -\frac{1}{r}\log (1+r)+\frac{\left(-r-\log (1-r)\right)}{2r}\left(-\frac{1}{r}\log (1-r)+\frac{1}{r}\log\left(1+r\right)+\frac{2}{1+r}\right)\\
 &&+2\left(\frac{r^2}{1-r}+r+2+\frac{2}{r}\log(1-r)\right)=0.\eeas 
 The number $r_1$ is the best possible.
\end{theo}
\begin{proof} Let $g(z)=f'(z)$ with $g(z)=\sum_{n=0}^\infty b_n z^n$ for $z\in\D$. Then, we have $b_n=(n+1)a_{n+1}$ for $n\geq 0$.
From (\ref{j2}), we have
\beas \mathcal{C}g(z):&=&\sum_{k=0}^\infty \frac{1}{k+1}\left(\sum_{n=0}^k b_n\right) z^k=\int_{0}^1\frac{g(t z)}{1-t z} d t,\\[2mm]{\it i.e., }\quad
\mathcal{C}f'(z):&=&\sum_{k=0}^\infty \frac{1}{k+1}\left(\sum_{n=0}^k (n+1)a_{n+1}\right) z^k=\int_{0}^1\frac{f'(t z)}{1-t z} d t.\eeas
From (\ref{j3}), we deduce that
\beas \phi_1(r)&=&\frac{1}{r}\int_{0}^{r}\frac{x}{1-x}dx=\frac{-r-\log(1-r)}{r}\\[2mm]\text{and}\quad
 \phi_2(r)&=&\frac{1}{r}\int_{0}^{r}\frac{x^2}{1-x}dx=-\frac{1}{r}\log(1-r)-1-\frac{r}{2}.\eeas
 Since $f(z)=\sum_{n=0}^\infty a_n z^n$ is analytic in $\D$ and $|f(z)|\leq 1$ in $\D$, we have $|a_n|\leq 1-|a_0|^2$ for $n\geq 1$.
 Therefore,
\beas\sum_{k=2}^{\infty}\left|a_k\right|\phi_k(r)\leq \left(1-a^2\right)\sum_{k=2}^{\infty}\phi_k(r)&=&\frac{\left(1-a^2\right)}{r}\int_{0}^{r}\frac{x^2}{(1-x)^2}dx\\[2mm]
&=&\left(1-a^2\right)\left(\frac{r^2}{1-r}-2\phi_2(r)\right).\eeas
In view of \textrm{Lemmas \ref{lem1}} and \ref{lem2}, we have 
\bea \label{j5}&&\left|\mathcal{C}f(z)\right|+\left|\mathcal{C}f'(z)\right|\phi_1(r)\nonumber\\[2mm]
&=&\left|\int_{0}^1\frac{f(t z)}{1-t z} d t\right|+\left|\int_{0}^1\frac{f'(t z)}{1-t z} d t\right|\phi_1(r)\nonumber\\[2mm]
&\leq&\int_{0}^1\left[\frac{|f(t z)|}{1-tr}+\frac{\left(1-|f(tz)|^2\right)\phi_1(r)}{\left(1-tr\right)\left(1-t^2r^2\right)}\right]dt\nonumber\\[2mm]
&=&\int_{0}^1\left[|f(tz)|+\frac{\phi_1(r)}{\left(1-t^2r^2\right)}\left(1-|f(tz)|^2\right)\right]\frac{dt}{\left(1-tr\right)}\nonumber\\[2mm]
&\leq&\int_{0}^1\left[\frac{a+tr}{1+atr}+\frac{\phi_1(r)}{\left(1-t^2r^2\right)}\left(1-\left(\frac{a+tr}{1+atr}\right)^2\right)\right]\frac{dt}{\left(1-tr\right)}\\[2mm]
&=&\int_{0}^1\left[\frac{a+tr}{1+atr}+\left(1-\left(\frac{a+tr}{1+atr}\right)^2\right)\frac{-r-\log (1-r)}{r\left(1-t^2r^2\right)}\right]\frac{dt}{\left(1-tr\right)}\nonumber\eea
\beas&=&\frac{1}{r}\log \frac{1}{1-r}-\frac{1-a}{ar}\log (1+ar)\nonumber\\
&&+\frac{\left(1-a^2\right)\left(-r-\log (1-r)\right)}{r}\int_{0}^1\frac{dt}{(1-tr)\left(1+atr\right)^2}\nonumber\\[2mm]
&=&\frac{1}{r}\log \frac{1}{1-r}-\frac{1-a}{ar}\log (1+ar)\nonumber\\[2mm]
&&+\frac{\left(1-a^2\right)\left(-r-\log (1-r)\right)}{r\left(a+1\right)^2}\left[-\frac{1}{r}\log (1-r)+\frac{1}{r}\log\left(1+ar\right)+\frac{a(a+1)}{1+ar}\right],\nonumber\eeas
where the inequality (\ref{j5}) holds for 
\beas 0\leq \frac{\phi_1(r)}{\left(1-t^2r^2\right)}=\frac{\left(-r-\log (1-r)\right)}{\left(r(1-t^2r^2)\right)}\leq \frac{1}{2}\quad\text{for}\quad t\in[0,1].\eeas 
It is easy to see that 
\beas \frac{\pa}{\pa t}\frac{\left(-r-\log (1-r)\right)}{r(1-t^2r^2)}=\frac{\left(-r-\log (1-r)\right)}{r}\frac{2r^2t}{(1-t^2r^2)^2}\geq 0,\eeas 
which shows that $\left(-r-\log (1-r)\right)/\left(r(1-t^2r^2)\right)$ is a monotonically increasing function of $t\in[0,1]$ and thus, we have 
\beas0\leq \frac{\left(-r-\log (1-r)\right)}{r(1-r^2)}\leq \frac{1}{2}\quad\text{holds}\quad r\leq R(\approx 0.493411),\eeas
where $R$ is the positive root of the equation $2\left(-r-\log (1-r)\right)=r\left(1-r^2\right)$.
Therefore,
\beas \left|\mathcal{C}f(z)\right|+\left|\mathcal{C}f'(z)\right|\phi_1(r)+\sum_{k=2}^{\infty}\left|a_k\right|\phi_k(r)\leq \frac{1}{r}\log \frac{1}{1-r}+\psi(a,r),\eeas
where $\psi(a,r)=(1-a)\phi(a,r)$ with
\beas \phi(a,r)&=&\frac{\left(-r-\log (1-r)\right)}{r\left(a+1\right)}\left(-\frac{1}{r}\log (1-r)+\frac{1}{r}\log\left(1+ar\right)+\frac{a(a+1)}{1+ar}\right)\\
&&+\left(1+a\right)\left(\frac{r^2}{1-r}-2\phi_2(r)\right)-\frac{1}{ar}\log (1+ar).\eeas
It is not difficult to show that 
\bea\label{r1} \frac{r^2}{1-r}-2\phi_2(r)=\frac{r^2}{1-r}+r+2+\frac{2}{r}\log(1-r)\geq 0\quad\text{for}\quad r\in(0,1).\eea
Since $\log(1+r)-(r-r^2/2)$ and $\log(1-r)-(r+r^2/2)$ are increasing functions of $r\in[0,1]$, it follows that
\bea\label{r2}\frac{1}{r}\log(1+r)\geq 1-\frac{r}{2}\quad\text{and}\quad -\log(1-r)\geq r+\frac{r^2}{2}\quad\text{for}\quad r\in[0,1]\eea
Using (\ref{r1}), (\ref{r2}) and by differentiating partially $\phi(a,r)$ with respect to $a$, we have 
\bea\frac{\pa}{\pa a}\phi(a,r)&=&-\frac{\left(-r-\log (1-r)\right)}{r(a+1)^2}\left(-\frac{1}{r}\log (1-r)+\frac{1}{r}\log\left(1+ar\right)+\frac{a(a+1)}{1+ar}\right)\nonumber\\
&&+\frac{\left(-r-\log (1-r)\right)}{r\left(a+1\right)}\left(\frac{1}{1+a r}+\frac{1+2a+a^2 r}{(1+a r)^2}\right)+\frac{r^2}{1-r}+r+2\nonumber\\
&&+\frac{2}{r}\log(1-r)+\frac{1}{a^2 r}\log(1+ar)-\frac{1}{a(1+ar)}\nonumber\\[2mm]
&=&-\frac{\left(-r-\log (1-r)\right)}{r(a+1)^2}\left(-\frac{1}{r}\log (1-r)+\frac{1}{r}\log\left(1+ar\right)\right)\nonumber\\
&&+\frac{(-r-\log (1-r))(2+a+ar)}{r(a+1)(1+ar)^2}+\frac{r^2}{1-r}+r+2+\frac{2}{r}\log(1-r)\nonumber\\
&&+\frac{1}{a^2 r}\log(1+ar)-\frac{1}{a(1+ar)}\nonumber\\[2mm]
\label{r3}&=&A_1(a,r)+A_2(a,r)+A_3(a,r)+\frac{r^2}{1-r}+r+2+\frac{2}{r}\log(1-r),\eea
where
\beas A_1(a,r)&=&\frac{1}{a^2 r}\log(1+ar)-\frac{1}{a(1+ar)}\geq \frac{1}{a}-\frac{r}{2}-\left(\frac{1}{a}-\frac{r}{1+ar}\right)=\frac{r(1-ar)}{2(1+ar)},\\[2mm]
A_2(a,r)&=&\frac{(-r-\log (1-r))(2+a+ar)}{r(a+1)(1+ar)^2}\quad\text{and}\\[2mm]
A_3(a,r)&=&-\frac{\left(-r-\log (1-r)\right)}{r(a+1)^2}\left(-\frac{1}{r}\log (1-r)+\frac{1}{r}\log\left(1+ar\right)\right).\eeas
The presence of two-variable terms in the expression $\frac{\pa}{\pa a}\phi(a,r)$ makes it challenging to demonstrate that it is greater than or equal to zero. Consequently, our 
objective is to express $\frac{\pa}{\pa a}\phi(a,r)$ in terms of a single variable.\\[2mm]
Let $B_1(a,r):=r(1-ar)/(2(1+ar))$. Differentiating partially $B_1(a,r)$ with respect to $a$, we obtain
\beas \frac{\pa }{\pa a}B_1(a,r)&=&-\frac{r^2(1-a r)}{2(1+a r)^2} - \frac{r^2}{2(1+a r)}\leq 0,\eeas 
which shows that $B_1(a,r)$ is a monotonically decreasing function of $a\in[0,1]$, and hence, we have 
\bea\label{r4} A_1(a,r)\geq B_1(a, r)\geq B_1(1,r)=\frac{r(1-r)}{2(1+r)}.\eea
Differentiating partially $A_2(a,r)$ with respect to $a$, we obtain
\beas \frac{\pa }{\pa a}A_2(a,r)&=&\frac{(1+r)(-r-\log(1-r))}{(1+a)r(1+a r)^2}-\frac{2(2+a+a r)(-r-\log(1-r))}{(1+a)(1+a r)^3}\\
&&-\frac{(2+a+a r)(-r-\log(1-r))}{(1+a)^2r(1+a r)^2}\\[2mm]
&=&-\frac{(2 a^2 r^2+2a^2 r+a r^2+7a r+3r+1)(-r -\log(1-r))}{(1+a)^2r(1+a r)^3}\leq 0.\eeas
Therefore, $A_2(a,r)$ is a monotonically decreasing function of $a\in[0,1]$, and it follows that 
\bea\label{r5} A_2(a,r)\geq A_2(1, r)=\frac{(3+r)(-r-\log(1-r))}{2r(1+r)^2}.\eea
Differentiating partially $A_3(a,r)$ with respect to $a$, we obtain 
\beas \frac{\pa }{\pa a}A_3(a,r)&=&-\frac{(-r-\log(1-r))}{(1+a)^2r(1+a r)} + \frac{2(-r -\log(1 - r))(-\log(1-r)+\log(1+a r))}{(1+a)^3 r^2}\\[2mm]
&=&\frac{(-r-\log(1-r))\left(-r-a r + (2+2 a r)\left( \log(1+a r)-\log(1- r)\right)\right)}{(1+a)^3 r^2 (1+a r)}.\eeas
From (\ref{r2}), we have 
\beas
(2+2 a r)\left( \log(1+a r)-\log(1- r)\right)&\geq& (2+2ar)\left(ar+r+\frac{r^2}{2}-\frac{a^2r^2}{2}\right)\\[2mm]
&=&2r+2 a r+r^2+2 a r^2+a^2 r^2+a r^3-a^3 r^3.\eeas
Therefore, we have 
\beas \frac{\pa }{\pa a}A_3(a,r)\geq \frac{(-r-\log(1-r))\left(r+a r+r^2+2 a r^2+a^2 r^2+a r^3-a^3 r^3\right)}{(1+a)^3 r^2 (1+a r)}\geq 0,\eeas
which shows that $A_3(a,r)$ is a monotonically increasing function of $a\in[0,1]$, and it follows that 
\bea\label{r6}A_3(a,r)\geq A_3(0, r)=\frac{(-r - \log(1-r))\log(1-r)}{r^2}.\eea
From (\ref{r3}), (\ref{r4}), (\ref{r5}) and (\ref{r6}), we have 
\beas \frac{\pa}{\pa a}\phi(a,r)&\geq &\frac{r(1-r)}{2(1+r)}+\frac{(3+r)(-r-\log(1-r))}{2r(1+r)^2}+\frac{(-r - \log(1-r))\log(1-r)}{r^2}\\
&&+\frac{r^2}{1-r}+r+2+\frac{2}{r}\log(1-r)\\[2mm]
&=&\frac{1}{2r^2(1-r)(1+r)^2}B_2(r),\eeas
where 
\beas B_2(r)&=&\left(r^2 + 9 r^3 - 3 r^5 + r^6\right)-\left(r-4 r^2+r^3+2 r^4\right)\log(1-r)\\
&&-\left(2+2 r-2 r^2-2 r^3\right) (\log(1-r))^2.\eeas
\begin{figure}[H]
\centering
\includegraphics[scale=0.7]{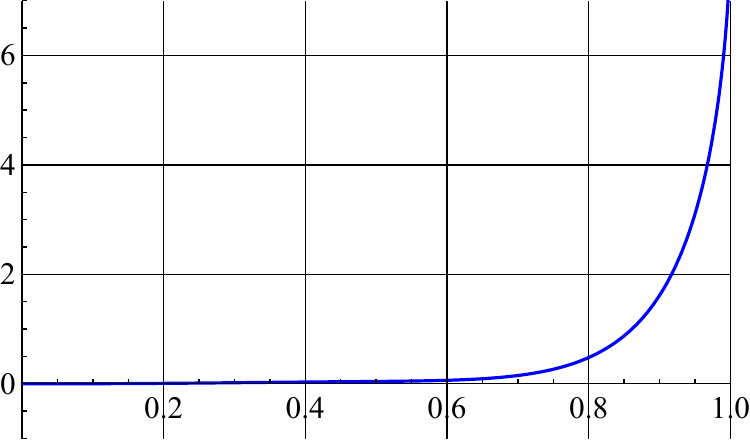}
\caption{The graph of $B_2(r)$ in $(0,1)$}
\label{fig1}
\end{figure}
Let $B_3(r)=r^2 + 9 r^3 - 3 r^5 + r^6$ and $B_4(r)=\left(r-4 r^2+r^3+2 r^4\right)\log(1-r)+\left(2+2 r-2 r^2-2 r^3\right) (\log(1-r))^2$.
\begin{figure}[H]
\centering
\includegraphics[scale=0.7]{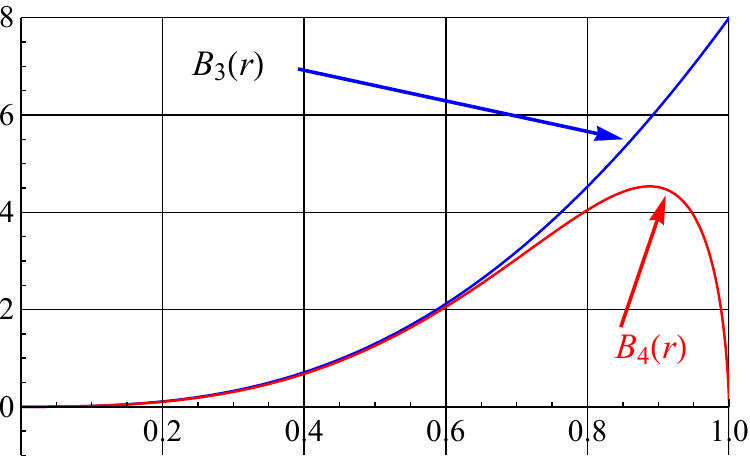}
\caption{The graphs of $B_3(r)$ and $B_4(r)$ in $(0,1)$}
\label{fig2}
\end{figure}
In Figures \ref{fig1} and \ref{fig2}, we have illustrated the graphical representation of $B_i(r)$ for $i=2,3,4$.
Differentiating $B_2(r)$ with respect to $r$, we have 
\beas B_2'(r)&=&3 r+24 r^2-2 r^3-15 r^4+6 r^5+(3+16 r +r^2 - 8 r^3) \log(1-r)\\
&&+(-2+4r+6r^2)(\log(1-r))^2.\eeas
\begin{figure}[H]
\centering
\includegraphics[scale=0.7]{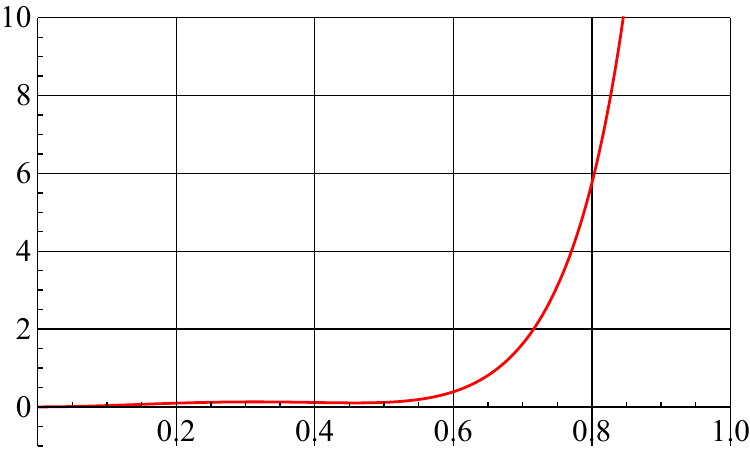}
\caption{The graph of $B_2'(r)$ in $(0,1)$}
\label{fig3}
\end{figure}
Using numerical computation, one can find that the equation $B_2'(r)=0$ has no root in $(0,1)$ with $B_2'(0)=0$ and $B_2'(1/2)=1/4 (26 - 41\log2 + (6\log2)^2)\approx 0.11592$. 
Therefore, $B_2'(r)\geq 0$ for $r\in[0,1]$, as illustrated in Figure \ref{fig3}, which shows that $B_2(r)$ is a monotonically increasing function of $r\in[0,1]$. Thus, we have $B_2(r)\geq B_2(0)=0$, {\it i.e.,} $\frac{\pa}{\pa a}\phi(a,r)\geq 0$.
Therefore, $\phi(a,r)$ is a monotonically increasing function of $a\in[0,1]$, and it follows that
\beas \phi(a,r)\leq \phi(1,r)=\phi(r),\eeas
where 
\bea\label{r8}\phi(r)&=&-\frac{\log (1+r)}{r}-\frac{1}{2}\left(-\frac{\log (1-r)}{r}+\frac{\log\left(1+r\right)}{r}+\frac{2}{1+r}\right)+2\left(\frac{r^2}{1-r}+r+2\right)\nonumber\\
&&+\frac{\log(1-r)}{r}\left(4+\frac{\log (1-r)}{2r}-\frac{\log\left(1+r\right)}{2r}-\frac{1}{1+r}\right).\eea
Differentiating $\phi(r)$ with respect to $r$, we have
\bea\label{r7} \phi'(r)&=&-\frac{1}{2}\left(\frac{1}{r^2}\log (1-r)+\frac{1}{r(1-r)}-\frac{1}{r^2}\log\left(1+r\right)+\frac{1}{r(1+r)}-\frac{2}{(1+r)^2}\right)\nonumber\\
&&+\left(-\frac{\log(1-r)}{r^2}-\frac{1}{r(1-r)}\right)\left(4+\frac{\log (1-r)}{2r}-\frac{\log (1+r)}{2r}-\frac{1}{1+r}\right)\nonumber\\
&&+\frac{\log(1-r)}{r}\left(-\frac{\log (1-r)}{2r^2}-\frac{1}{2r(1-r)}+\frac{\log\left(1+r\right)}{2r^2}-\frac{1}{2r(1+r)}\right.\nonumber\\
&&\left.+\frac{1}{(1+r)^2}\right)+2\left(\frac{2r}{1-r}+\frac{r^2}{(1-r)^2}+1\right)-\frac{1}{r(1+r)}+\frac{1}{r^2}\log (1+r)\nonumber\\[2mm]
&=& \frac{-5+7 r^2+6 r^3}{r(1-r)^2(1+r)^2}+\frac{1}{r^3}\log\left(1+r\right)\left(\frac{3}{2}r+\log(1-r)+\frac{r}{2(1-r)}\right)\nonumber\\
&&-\frac{1}{r^3}\log\left(1-r\right)\left(\frac{9}{2}r+\log(1-r)-\frac{r}{2(1+r)}+\frac{r}{1-r}-\frac{r^2}{(1+r)^2}\right).\eea
Let 
\beas g_1(r)&=&\frac{3}{2}r+\log(1-r)+\frac{r}{2(1-r)}\quad\text{and}\\
g_2(r)&=&\frac{9}{2}r+\log(1-r)-\frac{r}{2(1+r)}+\frac{r}{1-r}-\frac{r^2}{(1+r)^2}.\eeas
It is easy to see that 
\beas g_1'(r)=\frac{2 - 4 r + 3 r^2}{2 (1-r)^2}\geq 0\quad\text{and}\quad g_2'(r)=\frac{8 + 8 r - 3 r^2 - 17 r^3 + 11 r^4 + 9 r^5}{2(1-r)^2(1+r)^3}\geq 0\eeas
for $r\in[0,1]$. Using Sturm's theorem, it is easy to ascertain that the equation $8 + 8 r - 3 r^2 - 17 r^3 + 11 r^4 + 9 r^5=0$ has no positive real root.
Therefore, $g_1(r)$ and $g_2(r)$ are monotonically increasing functions of $r\in[0,1]$ and it follows that $g_1(r)\geq g_1(0)=0$ and $g_2(r)\geq g_2(0)=0$ for $r\in[0,1]$.
Using (\ref{r2}) in (\ref{r7}), we have 
\beas \phi'(r)&\geq&\frac{-5+7 r^2+6 r^3}{r(1-r)^2(1+r)^2}+\frac{(2-r)}{2r^2}\left(\frac{3}{2}r+\log(1-r)+\frac{r}{2(1-r)}\right)\nonumber\\
&&+\frac{(2+r)}{2r^2}\left(\frac{9}{2}r+\log(1-r)-\frac{r}{2(1+r)}+\frac{r}{1-r}-\frac{r^2}{(1+r)^2}\right)\\[2mm]
&=&\frac{4+5 r-8 r^2+2 r^3+10 r^4+3 r^5}{2r(1-r)^2(1+r)^2} + \frac{2}{r^2}\log(1 - r)\\
&=&\frac{1}{r^2}g_3(r),\eeas
where $g_3(r)=\left(4r+5 r^2-8 r^3+2 r^4+10 r^5+3 r^6\right)/\left(2(1-r)^2(1+r)^2\right) + 2\log(1- r)$. Differentiating $g_3(r)$ with respect to $r$, we have 
\beas g_3'(r)=\frac{3 r - 2 r^2 + 13 r^3 + 19 r^4 + 7 r^5 - 5 r^6 - 3 r^7}{(1-r)^3(1+r)^3}\geq 0\quad\text{for}\quad r\in[0,1].\eeas
Therefore, $g_3(r)$ is a monotonically increasing function of $r\in[0,1)$ and it follows that $g_3(r)\geq g_3(0)=0$ for $r\in[0,1]$.
Hence, we have $\phi'(r)\geq 0$ for $r\in[0,1]$ and it follows that $\phi(r)$ is a monotonically increasing function of $r\in[0,1]$. It is evident that 
\beas \lim_{r\to 0^+}\phi(r)=-1\quad\text{and}\quad \lim_{r\to 1^-}\phi(r)=\infty.\eeas   
Therefore, the equation $\phi(r)=0$ has exactly one root $r_1\in[0,1]$. 
We claim that $r_1\leq R (\approx 0.493411)$. For $r>R$, we have 
\bea\label{r9}\frac{\left(-r-\log (1-r)\right)}{r(1-r^2)}> \frac{1}{2},\quad{i.e.,}\quad\frac{1}{r}\log\frac{1}{1-r}>\frac{3-r^2}{2}.\eea
From (\ref{r2}), (\ref{r8}) and (\ref{r9}), we have
\beas \phi(r)&=&\frac{1}{2}\left(-1+\frac{1}{r}\log\frac{1}{1-r}\right)\left(-\frac{1}{r}\log (1-r)+\frac{1}{r}\log\left(1+r\right)+\frac{2}{1+r}\right)\\
&&+2\left(\frac{r^2}{1-r}+r+2+\frac{2}{r}\log(1-r)\right)-\frac{1}{r}\log (1+r)\\[2mm]
&>&-\frac{1}{r}\log (1+r)+\frac{1-r^2}{4}\left(\frac{3-r^2}{2}+\frac{1}{r}\log\left(1+r\right)+\frac{2}{1+r}\right)\\
&&+2\left(\frac{r^2}{1-r}+r+2+\frac{2}{r}\log(1-r)\right)\\[2mm]
&=&\frac{3+r^2}{4r}\left(\frac{(1-r^2)r}{3+r^2}\left(\frac{3-r^2}{2}+\frac{2}{1+r}\right)-\log (1+r)\right)\\
&&+2\left(\frac{r^2}{1-r}+r+2+\frac{2}{r}\log(1-r)\right).\eeas
Let
\beas g_4(r)&=&\frac{(1-r^2)r}{3+r^2}\left(\frac{3-r^2}{2}+\frac{2}{1+r}\right)-\log (1+r).\eeas
Differentiating $g_4(r)$ with respect to $r$, we see that
\beas g_4'(r)&=&\frac{39-45 r-7 r^2+43 r^3+13 r^4 - 11 r^5+3 r^6-3 r^7}{2(1-r)(3+r^2)^2}.\eeas
By employing Sturm's theorem, it is a straightforward process to ascertain that the equation $F(r):=39-45 r-7 r^2+43 r^3+13 r^4 - 11 r^5+3 r^6-3 r^7=0$ has exactly one positive real root, which lies in 
the interval $(1,2)$. This is evident from the fact that $F(1)=32$ and $F(2)=-71$. Therefore, $F(r)>0$ for $r\in[0,1]$, as illustrated in Figure \ref{fig4}.
\begin{figure}[H]
\centering
\includegraphics[scale=0.6]{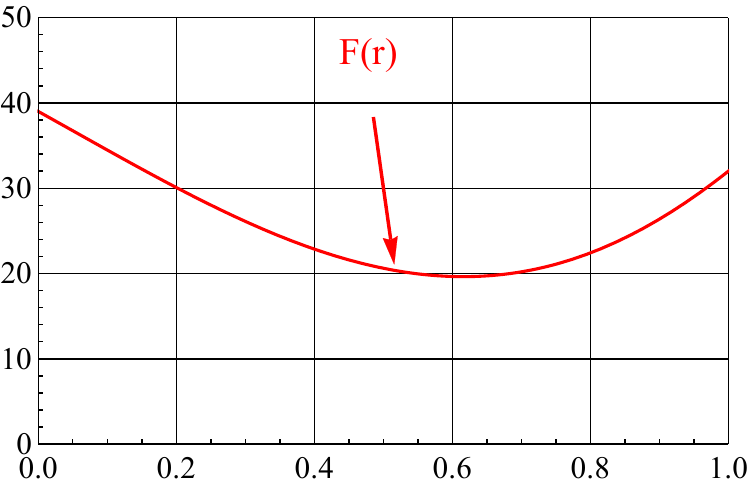}
\caption{The graph of the polynomial $F(r)$ in $[0,1]$}
\label{fig4}
\end{figure}
Thus, $g_4'(r)\geq 0$ for $r\in[0,1]$, 
which shows that $g_4(r)$ is a monotonically increasing function of $r\in(0,1)$ and it follows that $g_4(r)\geq g_4(0)=0$, {\it i.e.,} $\phi(r)>0$ for $r>R$.
Thus, we have $r_1\leq R$ and hence, we have $\psi(a,r)\leq 0$ for $0\leq r\leq r_1\leq R$,  where $r_1$ is the positive root of the equation $\phi(r)=0$ defined in (\ref{r8}). 
This completes the desired inequality (\ref{j1}). \\[2mm]
\indent To prove that the radius $r_1$ is optimal, let $a\in(0,1]$ and consider the function 
\beas f(z)=\frac{a+z}{1+az}=a+\left(1-a^2\right)\sum\limits_{n=1}^{\infty}(-a)^{n-1}z^n,\quad z\in\D.\eeas
Therefore, $f'(z)=\left(1-a^2\right)/(1+az)^2$ for $z\in\D$.
For the above function with $z=r$, we have
\beas 
\left|\mathcal{C}f(z)\right|&=&\int_{0}^1\frac{a+tr}{(1-tr)(1+atr)}dt=\frac{1}{r}\log\frac{1}{1-r}-\frac{1-a}{ar}\log(1+ar)\quad\text{and}\\[2mm]
\left|\mathcal{C}f'(z)\right|&=&\int_{0}^1\frac{1-a^2}{(1-tr)(1+atr)^2}dt\\[1.5mm]
&=&\frac{\left(1-a\right)}{\left(a+1\right)}\left(-\frac{\log (1-r)}{r}+\frac{\log (1+r)}{r}+\frac{a(a+1)}{1+ar}\right).\eeas
Since $a_0=a$ and $a_k=\left(1-a^2\right)(-a)^{k-1}$ for $k\geq 1$, we have
\beas \sum_{k=2}^{\infty}\left|a_k\right|\phi_k(r)&=&\left(1-a^2\right)\sum_{k=2}^{\infty}a^{k-1}\phi_k(r)\\
&=&\frac{1-a^2}{r}\sum_{k=2}^{\infty}a^{k-1}\int_{0}^r\frac{x^k}{1-x}dx\\
&=&\frac{1-a^2}{r}\int_{0}^r\frac{1}{1-x}\left(\sum_{k=2}^{\infty}a^{k-1}x^k\right)dx\\
&=&\frac{\left(1-a^2\right)a}{r}\int_{0}^r\frac{x^2}{(1-x)(1-ax)}dx\\
&=&\frac{\left(1+a\right)a}{r}\int_{0}^rx^2\left(\frac{1}{1-x}-\frac{a}{1-ax}\right)dx\\
&=& a\left(1+a\right)\phi_2(r)-(1+a)\phi_2(ar).\eeas
 Now
 \beas \phi_2(ar)&=&\phi_2\left(r-(1-a)r\right)=\sum_{k=2}^\infty \frac{\left(1-(1-a)\right)^k r^k}{k+1}\\
 &=&\sum_{k=2}^\infty \frac{\left(1-k(1-a)+O\left((1-a)^2\right)\right) r^k}{k+1}\\
 &=&\phi_2(r)-(1-a)\sum_{k=2}^\infty\frac{kr^k}{k+1}+O\left((1-a)^2r^2\right)\\
  &=&\phi_2(r)-(1-a)\sum_{k=2}^\infty\left(r^k-\frac{r^k}{k+1}\right)+O\left((1-a)^2r^2\right)\eeas
\beas &=&\phi_2(r)+(1-a)\phi_2(r)-(1-a)\frac{r^2}{1-r}+O\left((1-a)^2r^2\right)\\
 &=&(2-a)\phi_2(r)-(1-a)\frac{r^2}{1-r}+O\left((1-a)^2r^2\right).\eeas
 Thus, we have
 \beas \sum_{k=2}^{\infty}\left|a_k\right|\phi_k(r)&=&\left(a(1+a)-(1+a)(2-a)\right)\phi_2(r)+(1-a^2)\frac{r^2}{1-r}\\
 &&+(1+a)O\left((1-a)^2r^2\right)\\
&=&-2\left(1-a^2\right)\phi_2(r)+(1-a^2)\frac{r^2}{1-r}+(1+a)O\left((1-a)^2r^2\right).\eeas
Therefore,
\beas \left|\mathcal{C}f(z)\right|+\left|\mathcal{C}f'(z)\right|\phi_1(r)+\sum_{k=2}^{\infty}\left|a_k\right|\phi_k(r)= \frac{1}{r}\log \frac{1}{1-r}+(1-a)Q(a,r),\eeas
where
\beas Q(a,r)&=&\frac{\left(-r-\log (1-r)\right)}{r\left(a+1\right)}\left(-\frac{1}{r}\log (1-r)+\frac{1}{r}\log\left(1+ar\right)+\frac{a(a+1)}{1+ar}\right)\\
&&-2\left(1+a\right)\phi_2(r)+(1+a)\frac{r^2}{1-r}+(1+a)O\left((1-a)r^2\right)-\frac{\log (1+ar)}{ar}.\eeas
Letting $a\to 1^-$ in $Q(a,r)$, then we find that
\beas \lim_{a\to1^-} Q(a,r)&=&\frac{\left(-r-\log (1-r)\right)}{2r}\left(-\frac{1}{r}\log (1-r)+\frac{1}{r}\log\left(1+r\right)+\frac{2}{1+r}\right)\\
&&-4\phi_2(r)+\frac{2r^2}{1-r}-\frac{1}{r}\log (1+r)=\phi(r)>0\quad\text{for}\quad r>r_1.\eeas
This proves the sharpness of the inequality (\ref{j1}).\end{proof}
\begin{theo}\label{T2}
 If $f\in\mathcal{B}$ such that $f(z)=\sum_{n=0}^\infty a_n z^n$ for $z\in\D$ and $\omega\in\mathcal{B}_m$ for some $m\in\N$. Then, we have 
 \bea\label{j7}\left|\mathcal{C}f(\omega(z))\right|+\left|\mathcal{C}f'(\omega(z))\right|\phi_1(r^m)+\sum_{k=2}^{\infty}\left|a_k\right|\phi_k(r)\leq \frac{1}{r^m}\log \frac{1}{1-r^m}\eea
 for $|z|= r\leq R_{m,1}\leq R_m$, where $R_m\in(0,1)$ is the positive root of the equation 
 \beas 2\left(-r^m-\log (1-r^m)\right)=r^m\left(1-r^{2m}\right)\eeas 
 and $R_{m,1}\in(0,1)$ is the unique positive root of the equation
 \beas&&\frac{\left(-r^m-\log (1-r^m)\right)}{2r^m}\left\{-\frac{1}{r^m}\log (1-r^m)+\frac{1}{r^m}\log\left(1+r^m\right)+\frac{2}{1+r^m}\right\}\\
&&+2\left(\frac{r^2}{1-r}+r+2+\frac{2}{r}\log(1-r)\right)-\frac{1}{r^m}\log (1+r^m)=0.\eeas 
 The number $R_m$ is the best possible.
\end{theo}
\begin{proof} Suppose that $f\in\mathcal{B}$ and $\omega\in\mathcal{B}_m$. Let $|a_0|=a\in[0,1]$. In view of the Cauchy Schwarz lemma and \textrm{Lemma \ref{lem1}}, we have
\beas |f(v)|\leq \frac{a+|v|}{1+a|v|},\quad v\in\D\quad\text{and}\quad |\omega(z)|\leq |z|^m,\quad z\in\D.\eeas  
Therefore, we have
\beas |f(\omega(z))|\leq \frac{a+|\omega(z)|}{1+a|\omega(z)|}\leq \frac{a+|z|^m}{1+a|z|^m},\quad z\in\D.\eeas
By using similar arguments as in the proof of \textrm{Theorem \ref{T1}}, we have
\bea &&\left|\mathcal{C}f\left(\omega(z)\right)\right|+\left|\mathcal{C}f'\left(\omega(z)\right)\right|\phi_1(r^m)\nonumber\\
&=&\left|\int_{0}^1\frac{f(t \omega(z))}{1-t\omega(z)} d t\right|+\left|\int_{0}^1\frac{f'(t \omega(z))}{1-t \omega(z)} d t\right|\phi_1(r^m)\nonumber\\
&\leq&\int_{0}^1\left[\frac{|f\left(t \omega(z)\right)|}{1-tr^m}+\frac{\left(1-|f\left(t \omega(z)\right)|^2\right)\phi_1(r^m)}{\left(1-tr^m\right)\left(1-t^2r^{2m}\right)}\right]dt\nonumber\\
&=&\int_{0}^1\left[\left|f\left(t \omega(z)\right)\right|+\frac{\phi_1(r^m)}{\left(1-t^2r^{2m}\right)}\left(1-|f\left(t \omega(z)\right)|^2\right)\right]\frac{dt}{\left(1-tr^m\right)}\nonumber\\
\label{k1}&\leq&\int_{0}^1\left[\frac{a+tr^m}{1+atr^m}+\frac{\phi_1(r^m)}{\left(1-t^2r^{2m}\right)}\left\{1-\left(\frac{a+tr^m}{1+atr^m}\right)^2\right\}\right]\frac{dt}{\left(1-tr^m\right)}\\[2mm]
&=&\frac{1}{r^m}\log \frac{1}{1-r^m}-\frac{1-a}{ar^m}\log (1+ar^m)\nonumber\\
&&+\frac{\left(1-a^2\right)\left(-r^m-\log (1-r^m)\right)}{r^m\left(a+1\right)^2}\left(-\frac{\log (1-r^m)}{r^m}+\frac{\log\left(1+ar^m\right)}{r^m}+\frac{a(a+1)}{1+ar^m}\right),\nonumber\eea
where the inequality (\ref{k1}) hold for 
\beas 0\leq \frac{\phi_1(r)}{\left(1-t^2r^{2m}\right)}=\frac{\left(-r^m-\log (1-r^m)\right)}{r^m\left(1-t^2r^{2m}\right)}\leq \frac{1}{2}\quad{for}\quad t\in[0,1].\eeas 
Since $\left(-r^m-\log (1-r^m)\right)/\left(r^m(1-t^2r^{2m})\right)$ is a monotonically increasing function of $t$ and thus, we have 
\beas&& \frac{\left(-r^m-\log (1-r^m)\right)}{r^m(1-r^{2m})}\leq \frac{1}{2}\quad\text{for}\quad r\leq R_m\eeas
where $R_m\in(0,1)$ is the positive root of the equation $2\left(-r^m-\log (1-r^m)\right)=r^m\left(1-r^{2m}\right)$.
Therefore,
\beas\left|\mathcal{C}f\left(\omega(z)\right)\right|+\left|\mathcal{C}f'\left(\omega(z)\right)\right|\phi_1(r^m)+\sum_{k=2}^{\infty}\left|a_k\right|\phi_k(r)\leq \frac{1}{r^m}\log \frac{1}{1-r^m}+\Psi_m(a,r),\eeas
where $\Psi_m(a,r)=(1-a)\Phi_m(a,r)$ with
\beas\Phi_m(a,r)&=&\frac{\left(-r^m-\log (1-r^m)\right)}{r^m\left(a+1\right)}\left\{-\frac{1}{r^m}\log (1-r^m)+\frac{1}{r^m}\log\left(1+ar^m\right)+\frac{a(a+1)}{1+ar^m}\right\}\\
&&+\left(1+a\right)\left\{\frac{r^2}{1-r}-2\phi_2(r)\right\}-\frac{1}{ar^m}\log (1+ar^m).\eeas
Using (\ref{r1}), (\ref{r2}) and by differentiating partially $\Phi_m(a,r)$ with respect to $a$, we have 
\be\label{tr3}\frac{\pa}{\pa a}\Phi_m(a,r)=A_{1,m}(a,r)+A_{2,m}(a,r)+A_{3,m}(a,r)+\frac{r^2}{1-r}+r+2+\frac{2}{r}\log(1-r),\ee
where
\beas A_{1,m}(a,r)&=&\frac{1}{a^2 r^m}\log(1+ar^m)-\frac{1}{a(1+ar^m)}\\
&\geq&\frac{1}{a}-\frac{r^m}{2}-\left(\frac{1}{a}-\frac{r^m}{1+ar^m}\right)=\frac{r^m(1-ar^m)}{2(1+ar^m)},\\[2mm]
A_{2,m}(a,r)&=&\frac{(-r^m-\log (1-r^m))(2+a+ar^m)}{r^m(a+1)(1+ar^m)^2}\quad\text{and}\\[2mm]
A_{3,m}(a,r)&=&-\frac{\left(-r^m-\log (1-r^m)\right)}{r^m(a+1)^2}\left(-\frac{1}{r^m}\log (1-r^m)+\frac{1}{r^m}\log\left(1+ar^m\right)\right).\eeas
Using the method of proof of \textrm{Theorem \ref{T1}}, we obtain 
\bea\label{tr4} A_{1,m}(a,r)&\geq& \frac{r^m(1-r^m)}{2(1+r^m)},\\[2mm]
\label{tr5} A_{2,m}(a,r)&\geq&\frac{(3+r^m)(-r^m-\log(1-r^m))}{2r^m(1+r^m)^2}\\[2mm]\text{and}\quad
\label{tr6}A_{3,m}(a,r)&\geq&\frac{(-r^m - \log(1-r^m))\log(1-r^m)}{r^{2m}}.\eea
From (\ref{tr3}), (\ref{tr4}), (\ref{tr5}) and (\ref{tr6}), we have 
\beas \frac{\pa}{\pa a}\Phi_m(a,r)&\geq &\frac{r^m(1-r^m)}{2(1+r^m)}+\frac{(3+r^m)(-r^m-\log(1-r^m))}{2r^m(1+r^m)^2}\\
&&+\frac{(-r^m- \log(1-r^m))\log(1-r^m)}{r^{2m}}+\frac{r^2}{1-r}+r+2+\frac{2}{r}\log(1-r)\\[2mm]
&=&\frac{B_{2,m}(r)}{2r^{2m}(1-r)(1+r^m)^2},\eeas
where 
\beas B_{2,m}(r)=B_{3,m}(r)+B_{4,m}(r)\log\left(1-r\right)+B_{5,m}(r)\log(1-r^m)+B_{6,m}(r)(\log\left(1-r^m\right))^2\eeas
with 
\beas B_{3,m}(r)&=&r^{2 m}+r^{1+2m}+8r^{3 m}-4 r^{1+3 m}+4 r^{4 m}-2 r^{1+4 m}-r^{5 m}+r^{1+5 m},\\[2mm]
B_{4,m}(r)&=&4 r^{2m-1}+8 r^{3m-1}+4 r^{4 m-1}-4 r^{2 m}-8 r^{3 m}-4 r^{4 m},\\[2mm]
B_{5,m}(r)&=&-5 r^{m}+5 r^{1+m}-5 r^{2 m}+5 r^{1+2 m}-2 r^{3 m}+2 r^{1+3 m},\\[2mm]
B_{6,m}(r)&=&-2+2 r-4 r^{m}+4 r^{1+m}-2 r^{2 m}+2 r^{1+2 m}.\eeas
By numerical computation, one can find that for fixed $m\in\Bbb{N}$, the equation $B_{2,m}(r)=0$ has no root in $(0,1)$ with $\lim_{r\to 1^-} B_{2,m}(r)=8$ and $B_2(2,m)(0)=0$.
Therefore, $B_{2,m}(r)\geq 0$, {\it i.e.,} $\frac{\pa}{\pa a}\Phi_m(a,r)\geq 0$.
Hence, we have $\Phi_m(a,r)$ is a monotonically increasing function of $a\in[0,1]$ and thus
\beas \Phi_m(a,r)\leq \Phi_m(1,r)=\Phi_m(r),\eeas
where 
\bea\label{j6} \Phi_m(r)&=&\frac{\left(-r^m-\log (1-r^m)\right)}{2r^m}\left(-\frac{\log\left(1-r^m\right)}{r^m}+\frac{\log\left(1+r^m\right)}{r^m}+\frac{2}{1+r^m}\right)\nonumber\\
&&+2\left(\frac{r^2}{1-r}+r+2+\frac{2}{r}\log(1-r)\right)-\frac{1}{r^m}\log (1+r^m).\eea
We now proceed to rearrange the function $\Phi_m(r)$ as
 \beas\Phi_m(r)&=&-\frac{1}{2}\left(-\frac{1}{r^m}\log \left(1-r^m\right)+\frac{1}{r^m}\log\left(1+r^m\right)+\frac{2}{1+r^m}\right)\\
&&+\frac{1}{r^m}\log(1-r^m)\left(\frac{1}{2r^m}\log \left(1-r^m\right)-\frac{1}{2r^m}\log\left(1+r^m\right)-\frac{1}{1+r^m}\right)\\
&&+2\left(\frac{r^2}{1-r}+r+2+\frac{2}{r}\log(1-r)\right)-\frac{1}{r^m}\log\left(1+r^m\right).\eeas
Differentiating $\Phi_m(r)$ with respect to $r$, we have
\bea\label{t3} \Phi_m'(r)
&=& \frac{-4-m+6 r+2 m r-mr^2- 8 r^m-4r^{2 m}+12r^{1+m}+6 r^{1+2 m}}{r (1-r)^2\left(1+r^m\right)^2}\nonumber\\
&&+\frac{\log\left(1+r^m\right)}{r^{m+1}}\left(\frac{3m}{2}+\frac{m}{2(1-r^m)}+\frac{m}{r^m}\log(1-r^m)\right)\nonumber\\
&&-\frac{\log\left(1-r^m\right)}{r^{m+1}}\left(\frac{m}{2}+\frac{m}{1-r^m}-\frac{m}{2\left(1+r^m\right)}-\frac{mr^{m}}{\left(1+r^m\right)^2}\right.\nonumber\\
&&\left.+\frac{m\log\left(1-r^m\right)}{r^m}\right)-\frac{4}{r^2}\log(1-r).\eea
Let 
\beas g_{1,m}(r)&=&\frac{3mr^m}{2}+\frac{mr^m}{2(1-r^m)}+m\log(1-r^m)\quad\text{and}\\
g_{2,m}(r)&=&\frac{mr^m}{2}+\frac{mr^m}{1-r^m}-\frac{mr^m}{2\left(1+r^m\right)}-\frac{mr^{2m}}{\left(1+r^m\right)^2}+m\log\left(1-r^m\right).\eeas
It is easy to see that 
\beas&& g_{1,m}'(r)=\frac{m^2 r^{m-1} (2 - 4 r^m + 3 r^{2m})}{2(1-r^m)^2}\geq 0\\\text{and}
&&g_{2,m}'(r)=\frac{m^2 r^{3 m-1} (13 - r^m + 3 r^{2 m}+r^{3 m})}{2(1-r^m)^2 (1+r^m)^3}\geq 0\eeas
for $r\in[0,1]$ and $m\in\Bbb{N}$. Therefore, $g_{1,m}(r)$ and $g_{2,m}(r)$ are both monotonically increasing functions of $r\in[0,1]$ and it follows that 
\beas g_{1,m}(r)\geq g_{1,m}(0)=0\quad\text{and}\quad g_{2,m}(r)\geq g_{2,m}(0)=0\quad\text{for}\quad r\in[0,1].\eeas
Using (\ref{r2}) in (\ref{t3}), we have 
\beas \Phi_m'(r)&\geq&\frac{1}{r^{m+1}}\left(\frac{g_{3,m}(r)}{2 (1-r)^2(1-r^m)(1 + r^m)^2} + 2m\log(1- r^m)\right)=\frac{1}{r^{m+1}}g_{4,r}(r),\eeas
where 
\beas
g_{3,m}(r)&=&r^{5m} (m - 2 m r + m r^2) + r^{4 m} (-m + 2 m r - m r^2 - 4 r^3)\\
&& +r^{3 m} (m - 2 m r + m r^2 - 4 r^3)+ r^{2 m} (9 m - 18 m r + 9 m r^2 + 4 r^3)\\
&& +r^m (4 m - 8 m r + 4 m r^2 + 4 r^3).\eeas 
Let  $g_{4,m}(r)=(g_{3,m}(r)/(2 (1-r)^2(1-r^m)(1 + r^m)^2))+ 2m\log(1- r^m)$.
Differentiating $g_{4,m}(r)$ with respect to $r$, we have 
\beas g_{4,m}'(r)=\frac{g_{5,m}(r)}{r(1-r)^3(1-r^m)^2 (1+r^m)^3}\geq 0\quad\text{for}\quad r\in(0,1)\quad\text{and}\quad m\in\Bbb{N},\eeas
where
\beas g_{5,m}(r)&=&r^{7m}(-m^2 + 3 m^2 r - 3 m^2 r^2 + m^2 r^3) + r^{6 m} (6 r^3 + 2 m r^3 - 2 r^4 - 2 m r^4)\\
&&+ r^{5 m} (5 m^2 - 15 m^2 r + 15 m^2 r^2 + 6 r^3 + 2 m r^3 -5 m^2 r^3 - 2 r^4 - 2 m r^4)  \\
&&+ r^{4 m} (6 m^2 - 18 m^2 r + 18 m^2 r^2 - 12 r^3 - 4 m r^3 - 6 m^2 r^3 + 4 r^4 + 4 m r^4)\\
&&+r^{3 m} (m^2 - 3 m^2 r + 3 m^2 r^2 - 12 r^3 - 4 m r^3 - m^2 r^3 +4 r^4 + 4 m r^4)\\
&&+ r^{2 m} (3 m^2 - 9 m^2 r + 9 m^2 r^2 + 6 r^3 + 2 m r^3 - 3 m^2 r^3 -2 r^4 - 2 m r^4)\\
&&+ r^m (6 r^3 + 2 m r^3 - 2 r^4 - 2 m r^4).\eeas
Therefore, $g_{4,m}(r)$ is a monotonically increasing function of $r$ and it follows that $g_{4,m}(r)\geq g_{4,m}(0)=0$ for $r\in[0,1]$ and $m\in\Bbb{N}$.
Hence, we have $\Phi_m'(r)\geq 0$ for $r\in[0,1]$ and it follows that $\Phi_m(r)$ is a monotonically increasing function of $r$. Also, 
\beas \lim_{r\to 0^+}\Phi_m(r)=-1\quad\text{and}\quad \lim_{r\to 1^-}\Phi_m(r)=\infty.\eeas   
Therefore, the equation $\Phi_m(r)=0$ has exactly one root $R_{m,1}\in(0,1)$. Hence, $\Psi_m(a,r)\leq 0$ for $0\leq r\leq R_{m,1}$,  where $R_{m,1}$ is the positive root of the equation $\Phi_m(r)=0$ defined in (\ref{j6}).  \\[2mm]
\indent We claim that $R_{m,1}\leq R_m$. For $r>R_m$, we have 
\bea\label{j8}\frac{\left(-r^m-\log (1-r^m)\right)}{r^m(1-r^{2m})}> \frac{1}{2},\quad{i.e.,}\quad-1+\frac{1}{r^m}\log\frac{1}{1-r^m}>\frac{1-r^{2m}}{2}.\eea
From (\ref{r2}), (\ref{j6}) and (\ref{j8}), we have
\beas \Phi_m(r)&=&\frac{1}{2}\left(-1+\frac{1}{r^m}\log\frac{1}{1-r^m}\right)\left(-\frac{\log\left(1-r^m\right)}{r^m}+\frac{\log\left(1+r^m\right)}{r^m}+\frac{2}{1+r^m}\right)\\&&+2\left(\frac{r^2}{1-r}+r+2+\frac{2}{r}\log(1-r)\right)-\frac{1}{r^m}\log (1+r^m)\\[2mm]
&>&-\frac{1}{r^m}\log (1+r^m)+\frac{1-r^{2m}}{4}\left(\frac{3-r^{2m}}{2}+\frac{1}{r^m}\log\left(1+r^m\right)+\frac{2}{1+r^m}\right)\\
&&+2\left(\frac{r^2}{1-r}+r+2+\frac{2}{r}\log(1-r)\right)\\[2mm]
&=&\frac{3+r^{2m}}{4r^m} g_{6,m}(r)+2\left(\frac{r^2}{1-r}+r+2+\frac{2}{r}\log(1-r)\right),\eeas
where
\beas g_{6,m}(r)=\frac{(1-r^{2m})r^m}{3+r^{2m}}\left\{\frac{3-r^{2m}}{2}+\frac{2}{1+r^m}\right\}-\log(1-r^m).\eeas
Differentiating $g_{6,m}(r)$ with respect to $r$, we see that
\beas g_{6,m}'(r)=\frac{mr^{m-1}\left(39 - 45 r^m - 7 r^{2m}+43 r^{3m} + 13 r^{4 m}- 11 r^{5m}+3 r^{6m}- 3 r^{7m}\right)}{2(1-r^m)(3+r^{2m})^2}\geq 0\eeas
for $r\in[0,1]$ and $m\in\Bbb{N}$. Therefore, $g_{6,m}(r)$ is a monotonically increasing function of $r\in[0,1]$ and its follows that $g_{6,m}(r)\geq g_{6,m}(0)=0$, {\it i.e.,} $\Phi_m(r)>0$ for $r>R_m$.
Thus, we have $R_{m,1}\leq R_m$ and hence, we have $\Psi_m(a,r)\leq 0$ for $0\leq r\leq R_{m,1}\leq R_m$,  where $R_{m,1}$ is the positive root of the equation $\Phi_m(r)=0$ defined in (\ref{j6}). 
This completes the desired inequality (\ref{j7}). \\[2mm]
\indent To prove that the radius $R_{N,1}$ is optimal, let $a\in(0,1]$ and consider the functions 
\beas \omega(z)=z^m\quad\text{and}\quad f(z)=\frac{a+z}{1+az}=a+\left(1-a^2\right)\sum\limits_{n=1}^{\infty}(-a)^{n-1}z^n,\quad z\in\D.\eeas
Therefore, $f'(z)=\left(1-a^2\right)/(1+az)^2$ for $z\in\D$.
For the above function with $z=r$, we have
\beas 
\left|\mathcal{C}f(z^m)\right|&=&\int_{0}^1\frac{a+tr^m}{(1-tr^m)(1+atr^m)}dt=\frac{1}{r^m}\log\frac{1}{1-r^m}-\frac{(1-a)\log(1+ar^m)}{ar^m},\\[2mm]
\left|\mathcal{C}f'(z^m)\right|&=&\int_{0}^1\frac{1-a^2}{(1-tr^m)(1+atr^m)^2}dt\\[2mm]
&=&\frac{\left(1-a^2\right)}{\left(a+1\right)^2}\left(-\frac{1}{r^m}\log (1-r^m)+\frac{1}{r^m}\log\left(1+ar^m\right)+\frac{a(a+1)}{1+ar^m}\right).\eeas
Therefore,
\beas \left|\mathcal{C}f(z^m)\right|+\left|\mathcal{C}f'(z^m)\right|\phi_1(r^m)+\sum_{k=2}^{\infty}\left|a_k\right|\phi_k(r)= \frac{1}{r^m}\log \frac{1}{1-r^m}+(1-a)Q(a,r),\eeas
where
\beas Q(a,r)&=&\frac{\left(-r^m-\log (1-r^m)\right)}{r^m\left(a+1\right)}\left(-\frac{1}{r^m}\log (1-r^m)+\frac{1}{r^m}\log\left(1+ar^m\right)+\frac{a(a+1)}{1+ar^m}\right)\\
&&-2\left(1+a\right)\phi_2(r)+(1+a)\frac{r^2}{1-r}+(1+a)O\left((1-a)r^2\right)-\frac{\log (1+ar^m)}{ar^m}.\eeas
Letting $a\to 1^-$ in $Q(a,r)$, then we find that
\beas \lim_{a\to1^-} Q(a,r)&=&\frac{\left(-r^m-\log (1-r^m)\right)}{2r^m}\left(-\frac{1}{r^m}\log (1-r^m)+\frac{1}{r^m}\log\left(1+r^m\right)+\frac{2}{1+r^m}\right)\\
&&-4\phi_2(r)+\frac{2r^2}{1-r}-\frac{1}{r^m}\log (1+r^m)=\Phi_m(r)>0\quad\text{for}\quad r>R_{m,1}.\eeas 
This proves the sharpness part of the inequality (\ref{j7}).
\end{proof}
\section{Declarations:}
\noindent\textbf{Acknowledgment:}  The work of the second author is supported by University Grants Commission (IN) fellowship (No. F. 44-1/2018 (SA-III)).\\[1.5mm]
{\bf Conflict of Interest:} The authors declare that there are no conflicts of interest regarding the publication of this paper.\\[1.5mm]
\noindent\textbf{Authors contributions.} All the three authors have made equal contributions in reading, writing, and preparing the manuscript.\\
\noindent\textbf{Data availability statement:}  Data sharing is not applicable to this article as no datasets were generated or analyzed during the current study.

\end{document}